\UseRawInputEncoding

\documentclass[oneside, 11pt]{amsart} 

\usepackage{amsmath,amsthm,amssymb,epic}
\usepackage{eqlist,eqparbox}
\usepackage{amsfonts}
\usepackage{latexsym}
\usepackage[2emode]{psfrag}
\usepackage{amsthm}
\usepackage{amsmath}
\usepackage[all]{xy}

\newcommand{\Title}[1]{\bigskip\bigskip\centerline{\bf #1}\bigskip}
\newcommand{\Author}[1]{\medskip\centerline{ \it #1}}

\newcommand{\Affiliation}[1]{\medskip\centerline{#1}}
\newcommand{\Email}[1]{\medskip\centerline{#1}\bigskip}

\begin{document}

\newcommand{\N}{\mbox {$\mathbb N $}}
\newcommand{\Z}{\mbox {$\mathbb Z $}}
\newcommand{\Q}{\mbox {$\mathbb Q $}}
\newcommand{\R}{\mbox {$\mathbb R $}}
\newcommand{\lo }{\longrightarrow }
\newcommand{\ul}{\underleftarrow }
\newcommand{\rl}{\underrightarrow }
\newcommand{\rs }{\rightsquigarrow }
\newcommand{\ra }{\rightarrow }
\newcommand{\dd }{\rightsquigarrow }
\newcommand{\ol }{\overline }
\newcommand{\la }{\langle }
\newcommand{\tr }{\triangle }
\newcommand{\xr }{\xrightarrow }
\newcommand{\de }{\delta }
\newcommand{\pa }{\partial }
\newcommand{\LR }{\Longleftrightarrow }
\newcommand{\Ri }{\Rightarrow }
\newcommand{\va }{\varphi }
\newcommand{\Den}{{\rm Den}\,}
\newcommand{\Ker}{{\rm Ker}\,}
\newcommand{\Reg}{{\rm Reg}\,}
\newcommand{\Fix}{{\rm Fix}\,}
\newcommand{\Sup}{{\rm Sup}\,}
\newcommand{\Inf}{{\rm Inf}\,}
\newcommand{\Img}{{\rm Im}\,}
\newcommand{\Id}{{\rm Id}\,}
\newcommand{\ord}{{\rm ord}\,}

\newtheorem{theorem}{Theorem}[section]
\newtheorem{lemma}[theorem]{Lemma}
\newtheorem{proposition}[theorem]{Proposition}
\newtheorem{corollary}[theorem]{Corollary}
\newtheorem{definition}[theorem]{Definition}
\newtheorem{example}[theorem]{Example}
\newtheorem{examples}[theorem]{Examples}
\newtheorem{xca}[theorem]{Exercise}
\theoremstyle{remark}
\newtheorem{remark}[theorem]{Remark}
\newtheorem{remarks}[theorem]{Remarks}
\numberwithin{equation}{section}

\def\leftmark{L.C. Ciungu}

\Title{GENERALIZATIONS OF QUANTUM-WAJSBERG ALGEBRAS} 
\title[Generalizations of quantum-Wajsberg algebras]{}
                                                                           
\Author{\textbf{LAVINIA CORINA CIUNGU}}
\Affiliation{Department of Mathematics} 
\Affiliation{St Francis College}
\Affiliation{179 Livingston Street, Brooklyn, NY 11201, USA}
\Email{lciungu@sfc.edu}

\begin{abstract} 
Starting from involutive BE algebras, we redefine the pre-MV and meta-MV algebras, by introducing the notion of pre-Wajsberg and meta-Wajsberg algebras, as generalizations of quantum-Wajsberg algebras. 
We characterize these algebras, we show that any pre-Wajsberg algebra is a meta-Wajsberg algebra, and  
we prove that an implicative-orthomodular algebra is a quantum-Wajsberg algebra if and only if it is a meta-Wajsberg algebra. 
It is shown that the commutative quantum-Wajsberg (implicative-orthomodular, pre-Wajsberg, meta-Wajsberg)
algebras coincide with the Wajsberg algebras. 
We give conditions for the implicative-orthomodular, pre-Wajsberg and meta-Wajsberg algebras to be Wajsberg algebras.  The commutative center of an implicative-orthomodular algebra is defined and studied, and it is proved that 
it is a Wajsberg subalgebra of the implicative-orthomodular algebra. \\

\textbf{AMS classification (2020):} 06F35, 03G25, 06A06, 81P10, 06C15
\end{abstract}

\maketitle

\section{Introduction} 

In the last decades, the study of algebraic structures related to the logical foundations of quantum mechanics 
became a central topic of research. Generally known as quantum structures, these algebras serve as algebraic 
semantics for the classical and non-classical logics, as well as for the quantum logics. 
As algebraic structures connected with quantum logics we mention the following algebras: bounded involutive lattices,  De Morgan algebras, ortholattices, orthomodular lattices, MV algebras, quantum-MV algebras, orthomodular algebras. \\
The quantum-MV algebras (or QMV algebras) were introduced by R. Giuntini in \cite{Giunt1} 
as non-lattice generalizations of MV algebras (\cite{Chang, Cig1}) and as non-idempotent generalizations of orthomodular lattices (\cite{Beran, Kalm1}). 
These structures were intensively studied by R. Giuntini (\cite{Giunt2, Giunt3, Giunt4, Giunt5, Giunt6}), 
A. Dvure\v censkij and S. Pulmannov\'a (\cite{DvPu}), R. Giuntini and S. Pulmannov\'a (\cite{Giunt7}) and by 
A. Iorgulescu in \cite{Ior30, Ior31, Ior32, Ior33, Ior34, Ior35}. 
An extensive study on the orthomodular structures as quantum logics can be found in \cite{Ptak}. \\ 
Starting from the M algebras (RM algebras, ..., BE algebras, ..., BCK algebras) in the ``world" of 
algebras of logic, A. Iorgulescu introduced in \cite{Ior30} new commutative unital magmas with additional
operations (m-RM algebras, ..., m-BE algebras , ... m-BCK algebras), and she established connections between the
new algebras and the MV algebras, the ortholattices and the Boolean algebras (\cite{Ior30, Ior31, Ior34, Ior35}). 
Based on the involutive m-BE algebras, A. Iorgulescu redefined the quantum-MV algebras, and she generalized 
these algebras, by introducing three new algebras: orthomodular, pre-MV and meta-MV algebras. 
The implicative-ortholattices (also particular involutive BE algebras, introduced and studied in 
\cite{Ior30, Ior35}) are term-equivalent to ortholattices (redefined as particular involutive m-BE algebras in \cite{Ior30, Ior35}), and the implicative-Boolean algebras (introduced by A. Iorgulescu in 2009, also as particular involutive BE algebras, namely as particular involutive BCK algebras) are term-equivalent to Boolean algebras (redefined as particular involutive m-BE algebras, namely as particular (involutive) m-BCK algebras in \cite{Ior35}). \\
The connections between algebras of logic/algebras and quantum algebras were clarified by A. Iorgulescu in 
\cite{Ior30, Ior31, Ior35}: it was proved that the quantum algebras belong, in fact, to the ``world" of involutive unital commutative magmas. 
These connections were established by redefining equivalently the bounded involutive 
lattices and De Morgan algebras as involutive m-MEL algebras and the ortholattices, the MV, the quantum MV and 
the Boolean algebras as involutive m-BE algebras, verifying some properties, and then putting all of them on
the involutive ``big map". \\
We redefined in \cite{Ciu78} the quantum-MV algebras as involutive BE algebras, by introducing and studying the 
notion of quantum-Wajsberg algebras. It was proved that the quantum-Wajsberg algebras are equivalent to quantum-MV algebras and that the Wajsberg algebras are both quantum-Wajsberg algebras and commutative quantum-B algebras. 
We also redefined in \cite{Ciu81} the orthomodular algebras, by defining the implicative-orthomodular algebras, 
proving that the quantum-Wajsberg algebras are particular cases of these structures. 
We characterized the implicative-orthomodular algebras and we gave conditions for implicative-orthomodular algebras 
to be quantum-Wajsberg algebras.  
We also introduced and studied the notions of deductive systems in implicative-orthomodular algebras, and we defined  
the quotient implicative-orthomodular algebra with respect to the congruence induced by a deductive system. \\
In this paper, we redefine the pre-MV and the meta-MV algebras, by introducing the notions of pre-Wajsberg and meta-Wajsberg algebras, respectively, proving that these structures are generalizations of quantum-Wajsberg algebras. 
We give characterizations of these algebras, and we show that any pre-Wajsberg algebra is a meta-Wajsberg algebra. 
It is proved that an implicative-orthomodular algebra is a quantum-Wajsberg algebra if and only if it is a meta-Wajsberg algebra. 
If an involutive BE algebra is implicative, we prove that the implicative-orthomodular algebras coincide with the quantum-Wajsberg and the pre-Wajsberg algebras, and that any implicative-orthomodular algebra is a meta-Wajsberg algebra. 
We show that the Wajsberg algebras are implicative-orthomodular, pre-Wajsberg and meta-Wajsberg algebras, and 
that the commutative quantum-Wajsberg (implicative-orthomodular, pre-Wajsberg, meta-Wajsberg) algebras are Wajsberg algebras.  
We also prove that implicative-orthomodular, the pre-MV and the meta-MV algebras are Wajsberg algebras if and only if 
the relation $\le$ is antisymmetric. 
The commutative center of an implicative-orthomodular algebra $X$ is defined as the set of those elements of $X$ that commute with all other elements of $X$. 
We study certain properties of the commutative center and we prove that it is a Wajsberg subalgebra of $X$.

$\vspace*{1mm}$

\section{Preliminaries}

In this section, we recall some basic notions and results regarding BE algebras and m-BE algebras   
that will be used in the paper. Additionally, we prove new properties of involutive BE algebras. \\
\indent
Starting from the systems of positive implicational calculus, weak systems of positive implicational calculus 
and BCI and BCK systems, in 1966 Y. Imai and K. Is\`eki introduced the \emph{BCK algebras} (\cite{Imai}). 
BCK algebras are also used in a dual form, with an implication $\ra$ and with one constant element $1$, 
that is the greatest element (\cite{Kim2}). 
A (dual) BCK algebra is an algebra $(X,\ra,1)$ of type $(2,0)$ satisfying the following conditions, 
for all $x,y,z\in X$: 
$(BCK_1)$ $(x\ra y)\ra ((y\ra z)\ra (x\ra z))=1;$ 
$(BCK_2)$ $1\ra x=x;$ 
$(BCK_3)$ $x\ra 1=1;$ 
$(BCK_4)$ $x\ra y=1$ and $y\ra x=1$ imply $x=y$. 
In this paper, we use the dual BCK algebras. 
If $(X,\ra,1)$ is a BCK algebra, for $x,y\in X$ we define the relation $\le$ by $x\le y$ if and only if $x\ra y=1$, 
and $\le$ is a partial order on $X$. \\
\indent
\emph{Wajsberg algebras} were introduced in 1984 by Font, Rodriguez and Torrens in \cite{Font1} as algebraic model 
of $\aleph_0$-valued \L ukasiewicz logic.   
A \emph{Wajsberg algebra} is an algebra $(X,\ra,^*,1)$ of type $(2,1,0)$ satisfying the following conditions 
for all $x,y,z\in X$: 
$(W_1)$ $1\ra x=x;$ 
$(W_2)$ $(y\ra z)\ra ((z\ra x)\ra (y\ra x))=1;$ 
$(W_3)$ $(x\ra y)\ra y=(y\ra x)\ra x;$ 
$(W_4)$ $(x^*\ra y^*)\ra (y\ra x)=1$. 
Wajsberg algebras are bounded with $0=1^*$, and they are involutive. 
It was proved in \cite{Font1} that Wajsberg algebras are termwise equivalent to MV algebras. \\
\indent
\emph{BE algebras} were introduced in \cite{Kim1} as algebras $(X,\ra,1)$ of type $(2,0)$ satisfying the 
following conditions, for all $x,y,z\in X$: 
$(BE_1)$ $x\ra x=1;$ 
$(BE_2)$ $x\ra 1=1;$ 
$(BE_3)$ $1\ra x=x;$ 
$(BE_4)$ $x\ra (y\ra z)=y\ra (x\ra z)$. 
A relation $\le$ is defined on $X$ by $x\le y$ iff $x\ra y=1$. 
A BE algebra $X$ is \emph{bounded} if there exists $0\in X$ such that $0\le x$, for all $x\in X$. 
In a bounded BE algebra $(X,\ra,0,1)$ we define $x^*=x\ra 0$, for all $x\in X$. 
A bounded BE algebra $X$ is called \emph{involutive} if $x^{**}=x$, for any $x\in X$. \\
\indent
A \emph{suplement algebra} (\emph{S-algebra, for short}) (\cite{DvPu}) is an algebra $(X,\oplus,^*,0,1)$ of type $(2,1,0,0)$   
satisfying the following axioms, for all $x, y, z\in X:$ 
$(S_1)$ $x\oplus y=y\oplus x;$ 
$(S_2)$ $x\oplus (y\oplus z)=(x\oplus y)\oplus z;$ 
$(S_3)$ $x\oplus x^*=1;$ 
$(S_4)$ $x\oplus 0=x;$ 
$(S_5)$ $x^{**}=x;$ 
$(S_6)$ $0^*=1;$ 
$(S_7)$ $x\oplus 1=1$. \\
The following additional operations can be defined in a supplement algebra: \\
$\hspace*{3cm}$ $x\odot y=(x^*\oplus y^*)^*$, 
                $x\Cap_S y=(x\oplus y^*)\odot y$, 
                $x\Cup_S y=(x\odot y^*)\oplus y$. \\
Obviously, $(x\Cup_S y)^*=x^*\Cap_S y^*$ and $(x\Cap_S y)^*=x^*\Cup_S y^*$.  
For all $x,y\in X$, we define $\le_S$ by $x\le_S y$ iff $x=x\Cap_S y$. \\
\indent
A \emph{quantum-MV algebra} (\emph{QMV algebra, for short}) (\cite{Giunt1}) is an S-algebra $(X,\oplus,^*,0,1)$ satisfying the following axiom, for all $x, y, z\in X:$ \\
$(QMV)$ $x\oplus ((x^*\Cap_S y)\Cap (z\Cap_S x^*))=(x\oplus y)\Cap_S (x\oplus z)$. 

\begin{lemma} \label{qbe-10} $\rm($\cite{Ciu78}$\rm)$ 
Let $(X,\ra,1)$ be a BE algebra. The following hold for all $x,y,z\in X$: \\
$(1)$ $x\ra (y\ra x)=1;$ \\
$(2)$ $x\le (x\ra y)\ra y$. \\
If $X$ is bounded, then: \\
$(3)$ $x\ra y^*=y\ra x^*;$ \\
$(4)$ $x\le x^{**}$. \\
If $X$ is involutive, then: \\
$(5)$ $x^*\ra y=y^*\ra x;$ \\
$(6)$ $x^*\ra y^*=y\ra x;$ \\
$(7)$ $(x\ra y)^*\ra z=x\ra (y^*\ra z);$ \\
$(8)$ $x\ra (y\ra z)=(x\ra y^*)^*\ra z;$ \\   
$(9)$ $(x^*\ra y)^*\ra (x^*\ra y)=(x^*\ra x)^*\ra (y^*\ra y)$.  
\end{lemma}

\noindent
In a BE algebra $X$, we define the additional operation: \\
$\hspace*{3cm}$ $x\Cup y=(x\ra y)\ra y$. \\
If $X$ is involutive, we define the operations: \\
$\hspace*{3cm}$ $x\Cap y=((x^*\ra y^*)\ra y^*)^*$, 
                $x\odot y=(x\ra y^*)^*=(y\ra x^*)^*$, \\
and the relation $\le_Q$ by: \\
$\hspace*{3cm}$ $x\le_Q y$ iff $x=x\Cap y$. 

\begin{proposition} \label{qbe-20} $\rm($\cite{Ciu78}$\rm)$ Let $X$ be an involutive BE algebra. 
Then the following hold for all $x,y,z\in X$: \\
$(1)$ $x\le_Q y$ implies $x=y\Cap x$ and $y=x\Cup y;$ \\
$(2)$ $\le_Q$ is reflexive and antisymmetric; \\
$(3)$ $x\Cap y=(x^*\Cup y^*)^*$ and $x\Cup y=(x^*\Cap y^*)^*;$ \\ 
$(4)$ $x\le_Q y$ implies $x\le y;$ \\
$(5)$ $0\le_Q x \le_Q 1;$ \\
$(6)$ $0\Cap x=x\Cap 0=0$ and $1\Cap x=x\Cap 1=x;$ \\
$(7)$ $(x\Cap y)\ra z=(y\ra x)\ra (y\ra z);$ \\
$(8)$ $z\ra (x\Cup y)=(x\ra y)\ra (z\ra y);$ \\
$(9)$ $x\Cap y\le x,y\le x\Cup y;$ \\
$(10)$ $x\Cap (y\Cap x)=y\Cap x$ and $x\Cap (x\Cap y)=x\Cap y$. 
\end{proposition}

\begin{proposition} \label{qbe-30} $\rm($\cite{Ciu81}$\rm)$ Let $X$ be an involutive BE algebra. 
Then the following hold for all $x,y,z\in X$: \\
$(1)$ $x, y\le_Q z$ and $z\ra x=z\ra y$ imply $x=y;$ \emph{(cancellation law)} \\   
$(2)$ $(x\ra (y\ra z))\ra x^*=((y\ra z)\Cap x)^*;$ \\
$(3)$ $x\ra ((y\ra x^*)^*\Cup z)=y\Cup (x\ra z);$ \\
$(4)$ $((y\ra x)\Cap z)\ra x=y\Cup (z\ra x);$ \\
$(5)$ $x\le_Q y$ implies $(y\ra x)\odot y=x;$ \\ 
$(6)$ $x\ra (z\odot y^*)=((z\ra y)\odot x)^*;$ \\
$(7)$ $(x\Cup y)\Cap y=y$ and $(x\Cap y)\Cup y=y;$ \\
$(8)$ $z\Cap x=(x\ra (x\ra z)^*)^*;$ \\
$(9)$ $(x\Cap (y\Cap z))^*=((z\ra x)\Cap (z\ra y))\ra z^*;$ \\ 
$(10)$ $(x\Cap y)^*\ra (y\ra x)^*=y\Cup (y\ra x)^*$. 
\end{proposition}

\begin{definition} \label{qmv-30} $\rm($\cite{Ior30}$\rm)$
\emph{      
A \emph{(left-)m-BE algebra} is an algebra $(X,\odot,^{*},1)$ of type $(2,1,0)$ satisfying the following properties, 
for all $x,y,z\in X$:  
(PU) $1\odot x=x=x\odot 1;$ 
(Pcomm) $x\odot y=y\odot x;$ 
(Pass) $x\odot (y\odot z)=(x\odot y)\odot z;$  
(m-L) $x\odot 0=0;$ 
(m-Re) $x\odot x^{*}=0$, 
where $0:=1^*$. 
}\end{definition}
Note that, according to \cite[Cor. 17.1.3]{Ior35}, the involutive (left-)BE algebras $(X,\ra,^*,1)$ are 
term-equivalent to involutive (left-)m-BE algebras $(X,\odot,^*,1)$, by the mutually inverse transformations 
(\cite{Ior30, Ior35}): \\ 
$\hspace*{3cm}$ $\Phi:$\hspace*{0.2cm}$ x\odot y:=(x\ra y^*)^*$ $\hspace*{0.1cm}$ and  
                $\hspace*{0.1cm}$ $\Psi:$\hspace*{0.2cm}$ x\ra y:=(x\odot y^*)^*$. 

\begin{definition} \label{qbe-30-1} $\rm($\cite{Ior34},\cite[Def. 11.3.1]{Ior35}$\rm)$
\emph{
An involutive (left-)m-BE algebra $(X,\odot,^*,1)$ is called: \\
- an \emph{orthomodular algebra} (OM algebra for short), if it verifies: \\
$(Pom)$ $x\Cup (x\odot y)=x$; \\
- a \emph{pre-MV algebra} if it verifies: \\
$(Pmv)$ $x\odot (x^*\Cup y)=x\odot y$; \\
- a \emph{meta-MV algebra} if it verifies: \\
$(\Delta_m)$ $(x\Cap y)\odot (y\Cap x)^*=0$. 
}
\end{definition}


$\vspace*{1mm}$

\section{Generalizations of quantum-Wajsberg algebras}

In this section, we redefine the pre-MV and the meta-MV algebras, by introducing the notions of pre-Wajsberg and meta-Wajsberg algebras, respectively, proving that these structures are generalizations of quantum-Wajsberg algebras. 
We give characterizations of these algebras, and we show that any pre-Wajsberg algebra is a meta-Wajsberg algebra. 
It is proved that an implicative-orthomodular algebra is a quantum-Wajsberg algebra if and only if it is a meta-Wajsberg algebra. 
If an involutive BE algebra is implicative, we prove that the implicative-orthomodular algebras coincide with the quantum-Wajsberg and pre-Wajsberg algebras, and that any implicative-orthomodular algebra is a meta-Wajsberg algebra. \\
\noindent
A \emph{left-quantum-MV algebra}, or a \emph{(left-)QMV algebra} for short (\cite[Def. 3.10]{Ior34}), is an involutive (left-)m-BE algebra $(X,\odot,^{*},1)$ verifying the following axiom: for all $x,y,z\in X$, \\
(Pqmv) $x\odot ((x^*\Cup y)\Cup (z\Cup x^*))=(x\odot y)\Cup (x\odot z)$. \\
\noindent 
A \emph{left-quantum-Wajsberg algebra} (\emph{QW algebra, for short}) (\cite{Ciu78}) $(X,\ra,^*,1)$ is an 
involutive BE algebra $(X,\ra,^*,1)$ satisfying the following condition: for all $x,y,z\in X$, \\
(QW) $x\ra ((x\Cap y)\Cap (z\Cap x))=(x\ra y)\Cap (x\ra z)$. \\
Condition (QW) is equivalent to the following conditions: \\
($QW_1$) $x\ra (x\Cap y)=x\ra y;$ \\ 
($QW_2$) $x\ra (y\Cap (z\Cap x))=(x\ra y)\Cap (x\ra z)$. \\
Using the transformations $\Phi$ and $\Psi$, we can show that axioms $(Pqmv)$ and $(QW)$ are equivalent, hence 
the left-quantum-Wajsberg algebras are term-equivalent to left-quantum-MV algebras. \\
In what follows, by quantum-MV algebras and quantum-Wajsberg algebras we understand the left-quantum-MV algebras and  left-quantum-Wajsberg algebras, respectively. 
For more details on quantum-Wajsberg algebras we refer the reader to \cite{Ciu78}. 

\begin{definition} \label{gqw-10} 
\emph{
An involutive BE algebra $X$ is called: \\
- a \emph{pre-Wajsberg algebra} (pre-W algebra for short), if it verifies $(QW_1);$ \\
- an \emph{implicative-orthomodular algebra} (IOM algebra for short), if it verifies $(QW_2);$ \\
- a \emph{meta-Wajsberg algebra} (mreta-algebra for short), if it verifies: \\
$(QW_3)$ $(x\Cap y)\ra (y\Cap x)=1$. 
}
\end{definition}

\begin{remarks} \label{gqw-10-10} 
$(1)$ Condition $(QW_3)$ is equivalent to condition: for all $x,y\in X$, \\
$(QW_3^{'})$ $(x\Cup y)\ra (y\Cup x)=1$. \\
Indeed, by $(QW_3)$ we have: \\
$\hspace*{2cm}$ $(x\Cup y)\ra (y\Cup x)=(y\Cup x)^*\ra (x\Cup y)^*=(y^*\Cap x^*)\ra (x^*\Cap y^*)=1$. \\
Conversely, using $(QW_3^{'})$ we get: \\
$\hspace*{2cm}$ $(x\Cap y)\ra (y\Cap x)=(y\Cap x)^*\ra (x\Cap y)^*=(y^*\Cup x^*)\ra (x^*\Cup y^*)=1$. \\
$(2)$ Obviously, the quantum-Wajsberg algebras are pre-Wajsberg algebras and implicative-orthomodular algebras. 
\end{remarks}

\begin{proposition} \label{gqw-20} Any pre-Wajsberg algebra $X$ is a meta-Wajsberg algebra. 
\end{proposition}
\begin{proof} 
Assume that $X$ is a pre-Wajsberg algebra. Using $(QW_1)$, we get: \\
$\hspace*{2cm}$ $(x\Cup y)\ra (y\Cup x)=(x\Cup y)\ra ((y\ra x)\ra x)=(y\ra x)\ra ((x\Cup y)\ra x)$ \\
$\hspace*{5cm}$ $=(y\ra x)\ra (x^*\ra (x\Cup y)^*)=(y\ra x)\ra (x^*\ra (x^*\Cap y^*)$ \\
$\hspace*{5cm}$ $=(y\ra x)\ra (x^*\ra y^*)=(y\ra x)\ra (y\ra x)=1$, \\
that is condition $(QW_3^{'})$ is satisfied. Hence $X$ is a meta-Wajsberg algebra. 
\end{proof}

\begin{corollary} \label{gqw-20-10} Any quantum-Wajsberg algebra $X$ is a meta-Wajsberg algebra. 
\end{corollary}

\begin{lemma} \label{gqw-30} Let $X$ be an implicative-orthomodular algebra. Then $X$ satisfies the following  equivalent conditions: \\
$(IOM)$ $x\Cap (x^*\ra y)=x;$ $\hspace*{5cm}$ \\
$(IOM^{'})$ $x\Cap (y\ra x)=x;$ \\
$(IOM^{''})$ $x\Cup (x\ra y)^*=x$. 
\end{lemma}
\begin{proof} It is straightforward. 
\end{proof}

\begin{proposition} \label{gqw-30-10} Let $X$ be an involutive BE algebra satisfying the equivalent conditions 
$(IOM)$, $(IOM^{'})$, $(IOM^{''})$. Then the following hold for all $x,y,z\in X$: \\
$(1)$ $x\Cap (y\Cup x)=x$ and $x\Cup (y\Cap x)=x$. \\
$(2)$ if $x\le_Q y$, then $y\Cup x=y$ and $y^*\le_Q x^*;$ \\
$(3)$ if $x\le_Q y$, then $y\ra z\le_Q x\ra z$ and $z\ra x\le_Q z\ra y;$ \\ 
$(4)$ if $x\le_Q y$, then $x\Cap z\le_Q y\Cap z$ and $x\Cup z\le_Q y\Cup z;$ \\
$(5)$ $x\Cap ((y\ra x)\Cap (z\ra x))=x;$ \\
$(6)$ $x\Cup ((y^*\ra x^*)^*\Cup (z^*\ra x^*))=x;$ \\
$(7)$ $(x\Cup y)\ra (x\ra y)^*=y^*.$   
\end{proposition}
\begin{proof} 
$(1)$ Applying condition $(IOM^{'})$, we get: \\
$\hspace*{1.50cm}$ $x\Cap (y\Cup x)=x\Cap ((y\ra x)\ra x)=x;$ \\ 
$\hspace*{1.50cm}$ $x\Cup (y\Cap x)=x\Cup(y^*\Cup x^*)^*=(x^*\Cap ((y^*\ra x^*)\ra x^*))^*=(x^*)^*=x$. \\
$(2)$ Since $x=x\Cap y$, using $(1)$ we have $y\Cup x=y\Cup (x\Cap y)=y$. 
Thus $y^*=(y\Cup x)^*=y^*\Cap x^*$, that is $y^*\le_Q x^*$. \\
$(3)$ By $(IOM)$, $y\ra z\le_Q (y\ra z)^*\ra (y\ra x)^*$, so that we get: \\
$\hspace*{2.00cm}$ $(y\ra z)\Cap (x\ra z)=(y\ra z)\Cap ((x\Cap y)\ra z)$ \\
$\hspace*{5.15cm}$ $=(y\ra z)\Cap (z^*\ra (x^*\Cup y^*))$ \\
$\hspace*{5.15cm}$ $=(y\ra z)\Cap ((x^*\ra y^*)\ra (z^*\ra y^*))$ \\
$\hspace*{5.15cm}$ $=(y\ra z)\Cap ((y\ra x)\ra (y\ra z))$ \\
$\hspace*{5.15cm}$ $=(y\ra z)\Cap ((y\ra z)^*\ra (y\ra x)^*)=y\ra z$, \\
so that $y\ra z\le_Q x\ra z$. 
Since by $(2)$, $y=y\Cup x$, using $(IOM)$ we have $z\ra x\le_Q (z\ra x)^*\ra (y\ra x)^*$. 
Hence: \\
$\hspace*{2.00cm}$ $(z\ra x)\Cap (z\ra y)=(z\ra x)\Cap (z\ra (y\Cup x))$ \\
$\hspace*{5.15cm}$ $=(z\ra x)\Cap (z\ra ((y\ra x)\ra x))$ \\
$\hspace*{5.15cm}$ $=(z\ra x)\Cap ((y\ra x)\ra (z\ra x))$ \\
$\hspace*{5.15cm}$ $=(z\ra x)\Cap ((z\ra x)^*\ra (y\ra x)^*)$ \\
$\hspace*{5.15cm}$ $=z\ra x$. \\
It follows that $z\ra x\le_Q z\ra y$. \\
$(4)$ Since $x\le_Q y$ implies $y^*\le_Q x^*$, applying $(3)$ we get $x^*\ra z^*\le_Q y^*\ra z^*$ and 
$(y^*\ra z^*)\ra z^*\le_Q (x^*\ra z^*)\ra z^*$. 
It follows that $((x^*\ra z^*)\ra z^*)^*\le_Q ((y^*\ra z^*)\ra z^*)^*$, 
that is $x\Cap z\le_Q y\Cap z$. 
By $(3)$ we also have $y\ra z\le_Q x\ra z$ and $(x\ra z)\ra z\le_Q (y\ra z)\ra z$, hence $x\Cup z\le_Q y\Cup z$. \\
$(5)$ Since $x\le_Q y\ra x, z\ra x$, applying $(4)$, we have: 
$x=x\Cap (z\ra x)\le_Q (y\ra x)\Cap (z\ra x)$, hence $x\Cap ((y\ra x)\Cap (z\ra x))=x$. \\
$(6)$ By $(5)$, we have $x^*\Cup ((y\ra x)^*\Cup (z\ra x)^*))=x^*$. 
Replacing $x, y, z$ with $x^*, y^*, z^*$, respectively, it follows that 
$x\Cup ((y^*\ra x^*)^*\Cup (z^*\ra x^*))=x$. \\
$(7)$ Applying $(IOM^{'})$, we have $y=y\Cap (x\ra y)=((y^*\ra (x\ra y)^*)\ra (x\ra y)^*)^*$. 
It follows that: \\
$\hspace*{2.00cm}$ $y^*=(y^*\ra (x\ra y)^*)\ra (x\ra y)^*=((x\ra y)\ra y)\ra (x\ra y)^*$ \\
$\hspace*{2.50cm}$ $=(x\Cup y)\ra (x\ra y)^*$. 
\end{proof}

\begin{theorem} \label{gqw-60} An involutive BE algebra $X$ is an implicative-orthomodular algebra if and 
only if it satisfies the equivalent conditions $(IOM)$, $(IOM^{'})$, $(IOM^{''})$. 
\end{theorem}
\begin{proof} 
Let $X$ be an IOM algebra. 
Taking $y:=0$ in $(QW_2)$, we have $x^*=x^*\Cap (x\ra z)$, and replacing $x$ with $x^*$ and $z$ with $y$,  
we get $x=x\Cap (x^*\ra y)$. It follows that $X$ satisfies condition $(IOM)$, and by Lemma \ref{gqw-30}, 
conditions $(IOM^{'})$, $(IOM^{''})$ are also verified. 
Conversely, suppose that $X$ is an ivolutive BE algebra satisfying conditions $(IOM)$, $(IOM^{'})$, $(IOM^{''})$.  
Using Propositions \ref{qbe-30}$(3)$, \ref{gqw-30-10}$(6)$,$(7)$ we get: \\
$\hspace*{0.50cm}$ $x\ra (y\Cap (z\Cap x))=x\ra (y\Cap (z^*\Cup x^*)^*)=x\ra (y^*\Cup (x\ra (z^*\ra x^*)^*)^*$ \\
$\hspace*{3.50cm}$ $=x\ra (x\ra ((y^*\ra x^*)^*\Cup (z^*\ra x^*)^*))^*$ (Prop. \ref{qbe-30}$(3)$) \\
$\hspace*{3.50cm}$ $=(x\Cup ((y^*\ra x^*)^*\Cup (z^*\ra x^*)^*))\ra (x\ra ((y^*\ra x^*)^*\Cup (z^*\ra x^*)^*))^*$ \\
$\hspace*{12.50cm}$ (Prop. \ref{gqw-30-10}$(6)$) \\
$\hspace*{3.50cm}$ $=((y^*\ra x^*)^*\Cup (z^*\ra x^*)^*)^*$ (Prop. \ref{gqw-30-10}$(7)$) \\
$\hspace*{3.50cm}$ $=(y^*\ra x^*)\Cap (z^*\ra x^*)=(x\ra y)\Cap (x\ra z)$. \\ 
Hence axiom $(QW_2)$ is satisfied, that is $X$ is an implicative-orthomodular algebra. 
\end{proof}

\begin{remark} \label{gqw-60-10}
Due to Theorem \ref{gqw-60}, an implicative-orthomodular algebra can be also defined as an involutive BE algebra 
satisfying the equivalent conditions $(IOM)$, $(IOM^{'})$, $(IOM^{''})$ (see \cite{Ciu81}). 
\end{remark}

The following two propositions can be proved similarly as in  \cite[Prop. 3.6, 3.10]{Ciu81}. 

\begin{proposition} \label{gqw-50} 
Let $X$ be an IOM algebra. 
The following hold for all $x,y,z\in X$: \\
$(1)$ $(x\ra y)\Cup y=x\ra y;$ \\
$(2)$ $(x\ra y)\ra (y\Cap x)=x;$ \\
$(3)$ $x\ra (y\Cap x)=x\ra y;$ \\
$(4)$ $(z\ra y)\Cup (z\ra x)=z\ra y;$ \\ 
$(5)$ $(x\ra y)^*\Cap x=(x\ra y)^*;$ \\
$(6)$ $x\le y$ iff $y\Cap x=x;$ \\
$(7)$ $x\le_Q y$ and $y\le x$ imply $x=y;$ \\
$(8)$ $x\Cap y\le_Q y\le_Q x\Cup y;$ \\
$(9)$ $(x\Cup y)\ra y=x\ra y;$ \\
$(10)$ $x\Cap y, y\Cap x\le_Q x\ra y.$ 
\end{proposition}

\begin{proposition} \label{gqw-50-10} 
Let $X$ be an IOM algebra. 
The following hold for all $x,y,z\in X$: \\
$(1)$ $(x\Cap y)\Cap y=x\Cap y;$ \\
$(2)$ $x\Cup (y\Cap x)=x;$ \\ 
$(3)$ $x\Cap (y\Cup x)=x;$ \\
$(4)$ $x\Cap y\le_Q y\le_Q x\Cup y;$ \\
$(5)$ $(x\Cap y)\Cap (y\Cap z)=(x\Cap y)\Cap z;$ \\
$(6)$ $(x\Cup y)\Cup (y\Cup z)=(x\Cup y)\Cup z;$ \\
$(7)$ $\le_Q$ is transitive; \\
$(8)$ $(x\ra y)\Cup (x\ra (z\Cap y))=x\ra y;$ \\ 
$(9)$ $(x\ra y)\Cup ((z\ra x)\ra y)=x\ra y;$ \\
$(10)$ $(z\Cap x)\ra (y\Cap x)=(z\Cap x)\ra y;$ \\
$(11)$ $z\Cap ((y^*\ra z)\Cap (x^*\ra y))=z\Cap (x^*\ra y);$ \\
$(12)$ $x\Cup (x\ra y)^*=x;$ \\
$(13)$ $(z\Cup x)\ra (y\Cup x)=z\ra (y\Cup x)$. 
\end{proposition}
\begin{proof}
$(13)$ Using $(10)$, we get: 
$(z\Cup x)\ra (y\Cup x)=(y\Cup x)^*\ra (z\Cup x)^*=(y^*\Cap x^*)\ra (z^*\Cap x^*)=(y^*\Cap x^*)\ra z^*=
z\ra (y^*\Cap x^*)^*=z\ra (y\Cup x)$.
\end{proof}

\begin{proposition} \label{gqw-70} 
The following pairs of algebras are term-equivalent: \\
$(1)$ pre-Wajsberg algebras and pre-MV algebras; \\
$(2)$ implicative-orthomodular algebras and orthomodular algebras; \\
$(3)$ meta-Wajsberg algebras and meta-MV algebras.
\end{proposition}
\begin{proof} 
We use the mutually inverse transformations $\Phi$ and $\Psi$. \\
$(1)$ Assume that $X$ is a pre-Wajsberg algebra, that is it satisfies condition $(QW_1)$. 
Then we have: \\
$\hspace*{1.00cm}$ $x\ra (x\Cap y)=x\ra y$ iff $x\ra (x^*\Cup y^*)^*=x\ra (y^*)^*$ iff 
                   $(x\odot (x^*\Cup y^*))^*=(x\odot y^*)^*$. \\
Replacing $y$ by $y^*$, we get $x\odot (x^*\Cup y)=x\odot y$, that is $(Pmv)$, hence $X$ is a pre-MV algebra. 
Conversely, if is a pre-MV algebra, then by $(Pmv)$ we have: \\
$\hspace*{1.00cm}$ $x\odot (x^*\Cup y)=x\odot y$ iff $(x\ra (x^*\Cup y)^*)^*=(x\ra y^*)^*$ iff 
                   $x\ra (x\Cap y^*)=x\ra y^*$. \\
Replacing $y$ by $y^*$ we get $(QW_1)$, that is $X$ is a pre-Wajsberg algebra. \\
$(2)$ Suppose that $X$ is an implicative-orthomodular algebra, that is it satisfies condition $(IOM)$. 
It follows that $x^*\Cup (x^*\ra y)^*=x^*$, and replacing $x$ by $x^*$ and $y$ by $y^*$, we get 
$x\Cup (x\ra y^*)^*=x$, so that $x\Cup (x\odot y)=x$. It follows that $X$ satisfies condition $(Pom)$, hence 
it is an orthomodular algebra. 
Conversely, if $X$ satisfies $(Pom)$, then $x\Cup (x\odot y)=x$, so that $x\Cup (x\ra y^*)^*=x$, hence 
$x^*\Cap (x\ra y^*)=x^*$. Replacing $x$ by $x^*$ and $y$ by $y^*$, we get $x=x\Cap (x^*\ra y)$, 
that is condition $(IOM)$ is verified. It follows that $X$ is an implicative-orthomodular algebra. \\
$(3)$ If $X$ is a meta-Wajsberg algebra, then it satisfies $(QW_3)$. 
We can see that $(x\Cap y)\ra (y\Cap x)=1$ is equivalent to $(x\Cap y)\odot (y\Cap x)^*=0$, hence the meta-Wajsberg algebras are equivalent to meta-MV algebras. 
\end{proof}

\begin{theorem} \label{gqw-80} An involutive BE algebra $X$ is an IOM algebra if and only if it satisfies  
the following condition for all $x,y,z\in X$: \\
$(QW_2^{'})$ $x\ra (y\Cap (x\ra z)^*)=(x\ra y)\Cap (x\ra (x\ra z)^*)$.  
\end{theorem}
\begin{proof} 
Let $X$ be an involutive BE algebra.
We remark that, taking z:=1 in $(QW_2)$ and $z:=0$ in $(QW_2^{'})$, both conditions imply: \\
$(*)$ $x\ra (y\Cap x)=x\ra y$, for all $x,y\in X$. \\
Assume that $X$ satisfies condition $(QW_2)$. Then we have: \\
$\hspace*{1.00cm}$ $(x\ra y)\Cap (x\ra (x\ra z)^*)=x\ra (y\Cap ((x\ra z)^*\Cap x))$ (by $(QW_2)$) \\
$\hspace*{3.00cm}$ $=x\ra (y\Cap (x\ra (x\ra (x\ra z)^*)^*)^*)$ (Prop. \ref{qbe-30}$(8)$, for $z:=(x\ra z)^*$) \\
$\hspace*{3.00cm}$ $=x\ra (y\Cap (x\ra (z\Cap x))^*)$ (Prop. \ref{qbe-30}$(8)$) \\
$\hspace*{3.00cm}$ $=x\ra (y\Cap (x\ra z)^*)$  (by $(*))$, \\ 
hence condition $(QW_2^{'})$ is satisfied. 
Conversely, if $X$ satisfies $(QW_2^{'})$, we get: \\
$\hspace*{1.00cm}$ $x\ra (y\Cap (z\Cap x))=x\ra (y\Cap (x\ra (x\ra z)^*)^*)$ (Prop. \ref{qbe-30}$(8)$) \\ 
$\hspace*{4.00cm}$ $=(x\ra y)\Cap (x\ra (x\ra (x\ra z)^*)^*)$ (by $(QW_2^{'})$, for $z:=(x\ra z)^*$) \\
$\hspace*{4.00cm}$ $=(x\ra y)\Cap (x\ra (z\Cap x))$ (Prop. \ref{qbe-30}$(8)$) \\ 
$\hspace*{4.00cm}$ $=(x\ra y)\Cap(x\ra z)$ (by $(*)$). \\
Thus $X$ satisfies condition $(QW_2)$.  
\end{proof}

\begin{theorem} \label{gqw-90} An implicative-orthomodular algebra is a quantum-Wajsberg algebra if and only if 
it is a meta-Wajsberg algebra.  
\end{theorem}
\begin{proof} 
Let $X$ be an implicative-orthomodular algebra satisfying condition $(QW_3)$. It follows that: \\
$\hspace*{2.00cm}$ $1=(y\Cap x)\ra (x\Cap y)=(x\Cap y)^*\ra (y\Cap x)^*$ \\
$\hspace*{2.30cm}$ $=(x\Cap y)^*\ra ((y^*\ra x^*)\ra x^*)=(y^*\ra x^*)\ra ((x\Cap y)^*\ra x^*)$ \\
$\hspace*{2.30cm}$ $=(x\ra y)\ra (x\ra (x\Cap y))$, \\ 
hence $x\ra y\le x\ra (x\Cap y)$. 
Since $x\Cap y\le_Q y$ implies $x\ra (x\Cap y)\le_Q x\ra y$, by Proposition \ref{gqw-50}$(7)$ 
we get $x\ra (x\Cap y)=x\ra y$, hence $X$ satisfies condition $(QW_1)$. 
By definition, condition $(QW_2)$ is also satisfied, so that $X$ is a quantum-Wajsberg algebra. \\
Conversely, if $X$ is a quantum-Wajsberg algebra, then it is an implicative-orthomodular algebra, and by 
condition $(QW_1)$ we get: \\
$\hspace*{1.00cm}$ $(x\Cap y)\ra (y\Cap x)=(y\Cap x)^*\ra (x\Cap y)^*=(y\Cap x)^*\ra ((x^*\ra y^*)\ra y^*)$ \\ 
$\hspace*{4.00cm}$ $=(x^*\ra y^*)\ra ((y\Cap x)^*\ra y^*)=(y\ra x)\ra (y\ra (y\Cap x))$ \\
$\hspace*{4.00cm}$ $=(y\ra x)\ra (y\ra x)=1$, \\
hence condition $(QW_3)$ is verified. It follows that $X$ is a meta-Wajsberg algebra. 
\end{proof}

\begin{definition} \label{gqw-100} 
\emph{
A BE algebra is called \emph{implicative} if it satisfies the following condition: for all $x,y\in X$, \\
$(Pimpl)$ $(x\ra y)\ra x=x$. 
}
\end{definition}

\begin{lemma} \label{gqw-110} Let $(X,\ra,^*,1)$ be an involutive implicative BE algebra. 
The following hold for all $x,y\in X$. \\
$(1)$ 
$x^*\ra x=x$, or equivalently, $x\ra x^*=x^*;$ \\
$(2)$ 
$x\ra (x\ra y)=x\ra y;$ \\
$(3)$ $x\ra (y\ra x)^*=x^*;$ \\
$(4)$ $x\ra (y\ra x^*)=y\ra x^*;$ \\
$(5)$ $(y\ra x^*)\ra x=x$.  
\end{lemma}
\begin{proof} 
$(1)$ Taking $y:=0$ in $(Pimpl)$ we have $x^*\ra x=x$, and replacing $x$ by $x^*$ in this identity, 
we get $x\ra x^*=x^*$. \\
$(2)$ Using $(1)$, we have: $x\ra (x\ra y)=x\ra (y^*\ra x^*)=y^*\ra (x\ra x^*)=y^*\ra x^*=x\ra y$. \\
$(3)$ By $(Pimpl)$, we get: \\
$\hspace*{1cm}$ $(x\ra y)\ra x=x$ iff $x^*\ra (x\ra y)^*=x$ iff $x^*\ra (y^*\ra x^*)^*=x$, \\
and replacing $x$ by $x^*$ and $y$ by $y^*$, it follows that $x\ra (y\ra x)^*=x^*$. \\
$(4)$ Using $(1)$, we get $x\ra (y\ra x^*)=y\ra (x\ra x^*)=y\ra x^*;$ \\
$(5)$ By $(3)$, we have $(y\ra x^*)\ra x=x^*\ra (y\ra x^*)^*=x$. 
\end{proof}

\begin{theorem} \label{gqw-120} Let $(X,\ra,^*,1)$ be an involutive implicative BE algebra. 
Then the implicative-orthomodular algebra $(X,\ra,^*,1)$ is a quantum-Wajsberg algebra.  
\end{theorem}
\begin{proof} 
Let $X$ be an implicative-orthomodular algebra, that is $X$ satisfies conditions $(QW_2)$ and $(IOM^{'})$. 
By $(Pimpl)$, we get $(x^*\ra y)\ra x^*=x^*$, so that $(x\ra (x^*\ra y)^*)^*=x$. 
Taking into consideration this result, we get: \\
$\hspace*{1.80cm}$ $x\ra (x\Cap y^*)=(x\ra (x^*\ra y)^*)^*\ra (x^*\Cup y)^*$ \\
$\hspace*{4.00cm}$ $=x\ra ((x^*\ra y)\ra (x^*\Cup y)^*)$ (Lemma \ref{qbe-10}$(7)$) \\
$\hspace*{4.00cm}$ $=x\ra (((x^*\ra y)\ra y)\ra (x^*\ra y)^*)$ \\
$\hspace*{4.00cm}$ $=x\ra ((y^*\ra (x^*\ra y)^*)\ra (x^*\ra y)^*)$ \\
$\hspace*{4.00cm}$ $=x\ra (y^*\Cup (x^*\ra y)^*)$ \\
$\hspace*{4.00cm}$ $=x\ra (y\Cap (x^*\ra y))^*=x\ra y^*$ (by $(IOM^{'})$ \\
for all $x,y \in X$. Replacing $y$ by $y^*$ we get $x\ra (x\Cap y)=x\ra y$, that is $X$ satisfies condition $(QW_1)$. 
Hence $X$ is a quantum-Wajsberg algebra. 
\end{proof}

\begin{corollary} \label{gqw-130} If the involutive BE algebra $X$ is implicative, then the 
quantum-Wajsberg algebra coincides with the implicative-orthomodular algebra. 
\end{corollary}
\begin{proof} It follows by Remarks \ref{gqw-10-10}$(2)$ and Theorem \ref{gqw-120}. 
\end{proof}

\begin{corollary} \label{gqw-130-10} If the involutive BE algebra $X$ is implicative, then the 
implicative-orthomodular algebra is a pre-Wajsberg algebra and a meta-Wajsberg algebra. 
\end{corollary}
\begin{proof} It follows by Remarks \ref{gqw-10-10}$(2)$ and Corollary \ref{gqw-130}. 
\end{proof}

\begin{theorem} \label{gqw-140} Let $(X,\ra,^*,1)$ be an involutive implicative BE algebra. 
Then the pre-Wajsberg algebra $(X,\ra,^*,1)$ is an implicative-orthomodular algebra.  
\end{theorem}
\begin{proof} 
By hypothesis, $X$ satisfies conditions $(QW_1)$ and $(Pimpl)$. \\
By $(QW_1)$, $x\ra (x\Cap y)=x\ra y$, so that: \\
$\hspace*{2cm}$ $x\ra ((x^*\ra y^*)\ra y^*)^*=x\ra y$. \\ 
It follows that: \\
$\hspace*{2cm}$ $x\ra (y\ra (y\ra x)^*)^*=x\ra y$, \\
and replacing $y$ by $y\ra x^*$ we have: \\
$\hspace*{2cm}$ $x\ra ((y\ra x^*)\ra ((y\ra x^*)\ra x)^*)^*=x\ra (y\ra x^*)$. \\ 
By Lemma \ref{gqw-110}$(5)$,$(4)$, $(y\ra x^*)\ra x=x$, and we get
$x\ra (x\ra (y\ra x^*)^*)^*=y\ra x^*$. \\
Replacing $x$ by $x\ra y^*$ and $y$ by $x^*$ in the last identity, we have: \\
$\hspace*{2cm}$ $(x\ra y^*)\ra ((x\ra y^*)\ra (x^*\ra (x\ra y^*)^*)^*)^*=x^*\ra (x\ra y^*)^*$. \\
Since by $(Pimpl)$, $x^*\ra (x\ra y^*)^*=(x\ra y^*)\ra x=x$, we get 
$(x\ra y^*)\ra ((x\ra y^*)\ra x^*)^*=x$, so that $((x\ra y^*)\ra x^*)\ra (x\ra y^*)^*=x$, thus 
$(x\ra (x\ra y^*)^*)\ra (x\ra y^*)^*=x$. \\
Hence $x\Cup (x\ra y^*)^*=x$, for all $x,y\in X$, and replacing $y$ by $y^*$, we get $x\Cup (x\ra y)^*=x$. 
It follows that $X$ satisfies condition $(IOM^{''})$, that is $X$ is an implicative-orthomodular algebra. 
\end{proof}

\begin{corollary} \label{gqw-160} If the involutive BE algebra $X$ is implicative, then the 
implicative-orthomodular algebra coincides to the pre-Wajsberg algebra. 
\end{corollary}
\begin{proof} It follows by Corollary \ref{gqw-130-10} and Theorem \ref{gqw-140}. 
\end{proof}

\begin{remark} \label{gqw-160-10}
A. Iorgulescu introduced in \cite[Def. 3.2]{Ior32} the notion of \emph{(left-)orthomodular lattices} as involutive 
m-BE algebras satisfying conditions: \\
$(Pom)$ $(x\odot y)\oplus ((x\odot y)^*\odot x)=x$, or, equivalently, $x\Cup (x\odot y)=x;$ \\
$(m$-$Pimpl)$ $((x\odot y^*)^*\odot x^*)^*=x$. \\
We proved that condition $(Pom)$ is equivalent to condition $(QW_2)$, and similarly, we can show that condition $(m$-$Pimpl)$ is equivalent to condition $(Pimpl)$. 
We define the notion of \emph{implicative-orthomodular lattices} as involutive BE algebras 
satisfying conditions $(QW_2)$ and $(Pimpl)$, or, equivalently, implicative-orthomodular algebras satisfying 
condition $(Pimpl)$. 
As a consequence, the statements of Theorems \ref{gqw-120}, \ref{gqw-140} and Corollaries \ref{gqw-130}, 
\ref{gqw-130-10}, \ref{gqw-160} can be reformulated using the notion of implicative-orthomodular lattices. 
\end{remark}

Based on \cite[Examples 6.1, 6.2, 6.3, 6.4]{Ior34}, we give examples of quantum-Wajsberg, pre-Wajsberg, implicative-orthomodular and meta-Wajsberg algebras.  

\begin{example} \label{gqw-170} Let $X=\{0,a,b,c,d,1\}$ and let $(X,\ra,^*,1)$ be the involutive BE algebra 
with $\ra$ and the corresponding operation $\Cap$ given in the following tables:  
\[
\begin{array}{c|ccccccc}
\ra & 0 & a & b & c & d & 1 \\ \hline
0   & 1 & 1 & 1 & 1 & 1 & 1 \\ 
a   & d & 1 & c & 1 & 1 & 1 \\ 
b   & c & 1 & 1 & 1 & 1 & 1 \\ 
c   & b & 1 & a & 1 & c & 1 \\
d   & a & 1 & 1 & 1 & 1 & 1 \\
1   & 0 & a & b & c & d & 1
\end{array}
\hspace{10mm}
\begin{array}{c|ccccccc}
\Cap & 0 & a & b & c & d & 1 \\ \hline
0    & 0 & 0 & 0 & 0 & 0 & 0 \\ 
a    & 0 & a & b & c & d & a \\ 
b    & 0 & b & b & b & d & b \\ 
c    & 0 & a & b & c & d & c \\
d    & 0 & a & b & d & d & d \\
1    & 0 & a & b & c & d & 1
\end{array}
.
\]

We can see that $(X,\ra,^*,1)$ is a quantum-Wajsberg algebra. 
\end{example}

\begin{example} \label{gqw-180} Let $X=\{0,a,b,c,d,1\}$ and let $(X,\ra,^*,1)$ be the involutive BE algebra with $\ra$ and the corresponding operation $\Cap$ given in the following tables:  
\[
\begin{array}{c|ccccccc}
\ra & 0 & a & b & c & d & 1 \\ \hline
0   & 1 & 1 & 1 & 1 & 1 & 1 \\ 
a   & d & 1 & c & 1 & 1 & 1 \\ 
b   & c & 1 & 1 & 1 & 1 & 1 \\ 
c   & b & 1 & c & 1 & c & 1 \\
d   & a & 1 & 1 & 1 & 1 & 1 \\
1   & 0 & a & b & c & d & 1
\end{array}
\hspace{10mm}
\begin{array}{c|ccccccc}
\Cap & 0 & a & b & c & d & 1 \\ \hline
0    & 0 & 0 & 0 & 0 & 0 & 0 \\ 
a    & 0 & a & b & c & d & a \\ 
b    & 0 & b & b & b & d & b \\ 
c    & 0 & a & b & c & d & c \\
d    & 0 & a & b & b & d & d \\
1    & 0 & a & b & c & d & 1
\end{array}
.
\]

One can check that $(X,\ra,^*,1)$ is a pre-Wajsberg algebra. 
Since $a\ra (0\Cap(b\Cap a))=d\neq b=(a\ra 0)\Cap (a\ra b)$, condition $(QW_2)$ is not satisfied, hence 
$X$ is not an implicative-orthomodular algebra.   
\end{example}

\begin{example} \label{gqw-190} Let $X=\{0,a,b,c,d,1\}$ and let $(X,\ra,^*,1)$ be the involutive BE algebra with $\ra$ and the corresponding operation $\Cap$ given in the following tables:  
\[
\begin{array}{c|ccccccc}
\ra & 0 & a & b & c & d & 1 \\ \hline
0   & 1 & 1 & 1 & 1 & 1 & 1 \\ 
a   & d & 1 & 1 & 1 & d & 1 \\ 
b   & c & 1 & 1 & 1 & 1 & 1 \\ 
c   & b & 1 & 1 & 1 & 1 & 1 \\
d   & a & 1 & 1 & 1 & 1 & 1 \\
1   & 0 & a & b & c & d & 1
\end{array}
\hspace{10mm}
\begin{array}{c|ccccccc}
\Cap & 0 & a & b & c & d & 1 \\ \hline
0    & 0 & 0 & 0 & 0 & 0 & 0 \\ 
a    & 0 & a & b & c & d & a \\ 
b    & 0 & a & b & c & d & b \\ 
c    & 0 & a & b & c & d & c \\
d    & 0 & 0 & b & c & d & d \\
1    & 0 & a & b & c & d & 1
\end{array}
.
\]

Then $(X,\ra,^*,1)$ is an implicative-orthomodular algebra. 
Since $d\ra (d\Cap a)=a\neq 1=d\ra a$, condition $(QW_1)$ is not satisfied, so that $X$ is not a QW or 
a pre-W algebra.  
Moreover, $(a\Cap d)\ra (d\Cap a)=a\neq 1$, that is condition $(QW_3)$ is not verified, thus $X$ is not a 
meta-W algebra. 
\end{example}

\begin{example} \label{gqw-200} Let $X=\{0,a,b,c,d,1\}$ and let $(X,\ra,^*,1)$ be the involutive BE algebra with $\ra$ and the corresponding operation $\Cap$ given in the following tables:  
\[
\begin{array}{c|ccccccc}
\ra & 0 & a & b & c & d & 1 \\ \hline
0   & 1 & 1 & 1 & 1 & 1 & 1 \\ 
a   & d & 1 & 1 & 1 & d & 1 \\ 
b   & c & 1 & 1 & 1 & 1 & 1 \\ 
c   & b & 1 & c & 1 & 1 & 1 \\
d   & a & c & 1 & 1 & 1 & 1 \\
1   & 0 & a & b & c & d & 1
\end{array}
\hspace{10mm}
\begin{array}{c|ccccccc}
\Cap & 0 & a & b & c & d & 1 \\ \hline
0    & 0 & 0 & 0 & 0 & 0 & 0 \\ 
a    & 0 & a & b & c & 0 & a \\ 
b    & 0 & a & b & b & d & b \\ 
c    & 0 & a & b & c & d & c \\
d    & 0 & 0 & b & c & d & d \\
1    & 0 & a & b & c & d & 1
\end{array}
.
\]

One can show that $(X,\ra,^*,1)$ is a meta-Wajsberg algebra. 
We have $d\ra (d\Cap a)=a\neq c=d\ra a$, hence condition $(QW_1)$ is not satisfied, that is $X$ is not a QW or a 
pre-W algebra. 
Moreover, $d\ra (0\Cap (a\Cap d))=a\neq c=(d\ra 0)\Cap (d\ra a)$, so that condition $(QW_2)$ is not verified, thus  
$X$ is not an implicative-orthomodular algebra. 
\end{example}

$\vspace*{1mm}$

\section{Commutative generalizations of quantum-Wajsberg algebras} 

In this section, we define and study the commutativity property for the generalizations of quantum-Wajsberg algebras. 
We show that the Wajsberg algebras are implicative-orthomodular, pre-Wajsberg and meta-Wajsberg algebras, and 
that the commutative quantum-Wajsberg (implicative-orthomodular, pre-Wajsberg, meta-Wajsberg) algebras are Wajsberg algebras. 
We give conditions for the implicative-orthomodular algebras to be Wajsberg algebras, and we show that 
the implicative-orthomodular, the pre-MV and the meta-MV algebras are Wajsberg algebras if and only if 
the relation $\le$ is antisymmetric. \\
A BE algebra $X$ is called \emph{commutative} if $(x\ra y)\ra y=(y\ra x)\ra x$, for all $x,y\in X$. 
Obviously, any BCK algebra is a BE algebra, but the exact connection between BE algebras and 
BCK algebras is made in the papers \cite{Ior16, Ior17}: a BCK algebra is a BE algebra satisfying $(BCK_4)$ (antisymmetry) and $(BCK_1)$. \\
Since: \\
- commutative BE algebras are commutative BCK algebras (\cite{Walend1}]),  \\
- bounded commutative BCK are term-equivalent to MV algebras (\cite{Mund1}) and \\
- Wajsberg algebras are term-equivalent to MV algebras (\cite{Font1}), \\
it follows that bounded commutative BE algebras are bounded commutative BCK algebras. 
Hence the bounded commutative BE algebras are term-equivalent with MV algebras, thus with 
Wajsberg algebras. As a consequence, the commutative quantum-Wajsberg (implicative-orthomodular algebras, 
pre-Wajsberg algebras, meta-Wajsberg) algebras are Wajsberg algebras. 

\begin{lemma} \label{cgqw-10} If $(X,\ra,^*,1)$ is a Wajsberg algebra, then $\le_Q=\le$. 
\end{lemma}
\begin{proof} 
Let $x,y\in X$. If $x\le_Q y$, by Proposition \ref{qbe-20}$(4)$ we have $x\le y$. 
Conversely, if $x\le y$, then $x\ra y=1$, and we get: \\
$\hspace*{2.00cm}$ $x\Cap y=((x^*\ra y^*)\ra y^*)^*=((y^*\ra x^*)\ra x^*)^*=(x\ra (y^*\ra x^*)^*)^*$ \\
$\hspace*{3.00cm}$ $=(x\ra (x\ra y)^*)^*=(x\ra 1^*)^*=x^{**}=x$. \\
It follows that $x\le_Q y$, hence $\le_Q=\le$. 
\end{proof} 

\begin{proposition} \label{cgqw-20} $\rm($\cite{Ciu78}$\rm)$ Let $(X,\ra,0,1)$ be a Wajsberg algebra. 
The following hold for all $x,y,z\in X$: \\
$(1)$ $x\ra (x\Cap y)=x\ra y$ and $x\ra ((x\Cap y)\Cap x)=x\ra y;$ \\
$(2)$ $x\Cap (x^*\ra y)=x;$ \\
$(3)$ $(x\ra y)\ra (x\Cap y)=x;$ \\
$(4)$ $(z\Cap x)\ra (y\Cap x)=(z\Cap x)\ra y;$ \\
$(5)$ $(x\ra y)^*\Cap x=(x\ra y)^*;$ \\
$(6)$ $(x\Cap y)\Cap y=x\Cap y;$ \\
$(7)$ $x\Cap y\le_Q y \le_Q x\Cup y;$ \\
$(8)$ $(x\Cap y)\Cap (y\Cap z)=(x\Cap y)\Cap z;$ \\
$(9)$ $(x\Cap y)\Cap z=y\Cap (x\Cap z);$ \\
$(10)$ $x\ra (y\Cap z)=(x\ra y)\Cap (x\ra z)$.   
\end{proposition}

\begin{proposition} \label{cgqw-30} Wajsberg algebras are implicative-orthomodular algebras, pre-Wajsberg algebras 
and meta-Wajsberg algebras.  
\end{proposition}
\begin{proof} 
Let $(X,\ra,^*,1)$ be a Wajsberg algebra and let $x,y,z\in X$. 
By commutativity and applying Proposition \ref{cgqw-20}$(8)$,$(9)$,$(1)$,$(10)$, we get: \\
$\hspace*{2.10cm}$ $x\ra ((x\Cap y)\Cap (z\Cap x))=x\ra ((y\Cap x)\Cap (x\Cap z))=x\ra ((y\Cap x)\Cap z)$ \\
$\hspace*{6.00cm}$ $=x\ra (x\Cap (y\Cap z))=x\ra (y\Cap z)$ \\
$\hspace*{6.00cm}$ $=(x\ra y)\Cap (x\ra z)$. \\
Hence $X$ satisfies axiom $(QW)$, that is it is a quantum-Wajsberg algebra. 
By Remarks \ref{gqw-10-10} and Proposition \ref{gqw-20}, it follows that $X$ is an implicative-orthomodular algebra, 
as well as a pre-Wajsberg algebra and a meta-Wajsberg algebra.  
\end{proof}

\begin{theorem} \label{cgqw-40} A pre-Wajsberg algebra $(X,\ra,^*,1)$ is a Wajsberg algebra if and only if, 
for all $x,y\in X$, $x\le y$ implies $x\le_Q y$. 
\end{theorem}
\begin{proof} 
Let $X$ be a Wajsberg algebra, so that, by Proposition \ref{cgqw-30}, it follows that $X$ is a pre-Wajsberg algebra. 
Consider $x,y\in X$ such that $x\le y$, that is $x\ra y=1$. 
By commutativity, we have: 
$x=(x^*)^*=(1\ra x^*)^*=((x\ra y)\ra x^*)^*=((y^*\ra x^*)\ra x^*)^*=y\Cap x=x\Cap y$. 
Hence $x\le_Q y$. 
Conversely, suppose that $X$ is a pre-Wajsberg algebra such that $x\le y$ implies $x\le_Q y$ for all $x,y\in X$. 
Using $(QW_1)$, we get: \\
$\hspace*{2.00cm}$ $(x\Cap y)\ra (y\Cap x)=(y\Cap x)^*\ra (x\Cap y)^*=(y\Cap x)^*\ra ((x^*\ra y^*)\ra y^*)$ \\
$\hspace*{5.00cm}$ $=(x^*\ra y^*)\ra ((y\Cap x)^*\ra y^*)=(y\ra x)\ra (y\ra (y\Cap x))$ \\
$\hspace*{5.00cm}$ $=(y\ra x)\ra (y\ra x)=1$. \\
It follows that $x\Cap y\le y\Cap x$, hence $x\Cap y\le_Q y\Cap x$, and similarly $y\Cap x\le_Q x\Cap y$. 
Since $\le_Q$ is antisymmetric (Proposition \ref{qbe-20}$(2)$), we get $x\Cap y=y\Cap x$, that is $X$ is a 
Wajsberg algebra. 
\end{proof} 

\begin{corollary} \label{cgqw-40-10} A pre-Wajsberg algebra is a Wajsberg algebra if and only if the relation 
$\le$ is antisymmetric. 
\end{corollary}

\begin{corollary} \label{cgqw-40-20} A quantum-Wajsberg algebra is a Wajsberg algebra if and only if the relation 
$\le$ is antisymmetric. 
\end{corollary}

\begin{theorem} \label{cgqw-50} An implicative-orthomodular algebra $(X,\ra,^*,1)$ is a Wajsberg algebra 
if and only if it satisfies the following conditions, for all $x,y\in X$: \\
$(1)$ $x\le_Q x\Cup y;$ \\
$(2)$ $(x\Cup y)\ra x=y\ra x$.   
\end{theorem}
\begin{proof} 
If $X$ is a Wajsberg algebra, then by Proposition \ref{cgqw-30}, it is an implicative-orthomodular algebra. 
Since $X$ is commutative, applying Proposition \ref{gqw-50}$(8)$,$(9)$, $X$ satisfies conditions $(1)$ and $(2)$.  
Conversely, suppose $X$ is an implicative-orthomodular algebra satisfying conditions $(1)$ and $(2)$. 
By $(2)$, $x\ra y=(y\Cup x)\ra y$, hence $(x\ra y)\ra y=((y\Cup x)\ra y)\ra y$, that is $x\Cup y=(y\Cup x)\Cup y$. 
Similarly we get $y\Cup x=(x\Cup y)\Cup x$. 
Replacing $x$ by $x\Cup y$ and $y$ by $x$ in $(1)$, we get $x\Cup y\le_Q (x\Cup y)\Cup x=y\Cup x$, 
and similarly $y\Cup x\le_Q (y\Cup x)\Cup y=x\Cup y$. 
Since $\le_Q$ is antisymmetric, we get $x\Cup y=y\Cup x$, that is $X$ is commutative.
Hence $X$ is a Wajsberg algebra. 
\end{proof} 

\begin{proposition} \label{cgqw-60} A meta-Wajsberg algebra is a Wajsberg algebra if and only if the relation 
$\le$ is antisymmetric. 
\end{proposition}
\begin{proof}
Let $X$ be a meta-Wajsberg algebra such that $\le$ is antisymmetric, and let $x,y\in X$. 
Since by $(QW_3^{'})$, $(x\Cup y)\ra (y\Cup x)=1$ and $(y\Cup x)\ra (x\Cup y)=1$, we have $x\Cup y\le y\Cup x$ and 
$y\Cup x\le x\Cup y$, hence $x\Cup y=y\Cup x$. It follows that $X$ is commutative, hence it is a Wajsberg algebra. 
Conversely, if $X$ is a Wajsberg algebra, then by commutativity, $(x\Cup y)\ra (y\Cup x)=(x\Cup y)\ra (x\Cup y)=1$, 
for all $x,y\in X$. Thus $X$ satisfies condition $(QW_3^{'})$, so that it is a meta-Wajsberg algebra. 
\end{proof}

$\vspace*{1mm}$

\section{Commutative center of implicative-orthomodular algebras}

The commutative center of an implicative-orthomodular algebra $X$ is defined as the set of those elements of $X$ that commute with all other elements of $X$. 
We study certain properties of the commutative center and we prove that it is a Wajsberg subalgebra of $X$. 
In what follows, $(X,\ra,^*,1)$ will be an implicative-orthomodular algebra, unless otherwise stated.

\begin{definition} \label{ciom-10} 
\emph{
We say that the elements $x,y\in X$ \emph{commute}, denoted by $x \mathcal{C} y$, if $x\Cup y=y\Cup x$. 
The \emph{commutative center} of $X$ is the set $\mathcal{Z}(X)=\{x\in X\mid x \mathcal{C} y$, for all $y\in X\}$.  
}
\end{definition}

\noindent
Obviously $0,1\in \mathcal{Z}(X)$. 

\begin{proposition} \label{ciom-20} Let $x,y\in X$. The following are equivalent: \\
$(a)$ $x \mathcal{C} y;$ \\
$(b)$ $x\Cap y=y\Cap x;$ \\
$(c)$ $(x\ra y)\ra (x\Cap y)=x$.
\end{proposition}
\begin{proof}
$(a)\Rightarrow (b)$ If $x \mathcal{C} y$, then $x^*\Cap y^*=(x\Cup y)^*=(y\Cup x)^*=y^*\Cap x^*$. 
Using twice Proposition \ref{gqw-50}$(2)$, we get: \\
$\hspace*{2.00cm}$ $x^*\Cup y^*=(x^*\ra y^*)\ra y^*=(x^*\ra y^*)\ra ((y^*\ra x^*)\ra (x^*\Cap y^*))$ \\
$\hspace*{3.30cm}$ $=(y^*\ra x^*)\ra ((x^*\ra y^*)\ra (x^*\Cap y^*))$ \\
$\hspace*{3.30cm}$ $=(y^*\ra x^*)\ra ((x^*\ra y^*)\ra (y^*\Cap x^*))$ \\
$\hspace*{3.30cm}$ $=(y^*\ra x^*)\ra x^*=y^*\Cup x^*$. \\
It follows that $x\Cap y=(x^*\Cup y^*)^*=(y^*\Cup x^*)^*=y\Cap x$. \\
$(b)\Rightarrow (a)$ Assume that $x\Cap y=y\Cap x$. By Proposition \ref{gqw-50}$(2)$, we have: \\
$\hspace*{2.00cm}$ $x\Cup y=(x\ra y)\ra y=(x\ra y)\ra ((y\ra x)\ra (x\Cap y))$ \\
$\hspace*{3.00cm}$ $=(y\ra x)\ra ((x\ra y)\ra (x\Cap y))=(y\ra x)\ra ((x\ra y)\ra (y\Cap x))$ \\
$\hspace*{3.00cm}$ $=(y\ra x)\ra x=y\Cup x$. \\
$(b)\Rightarrow (c)$ Since $x\Cap y=y\Cap x$, we have $(x\ra y)\ra (x\Cap y)=(x\ra y)\ra (y\Cap x)=x$. \\
$(b)\Rightarrow (c)$ Using $(b)$ and Proposition \ref{gqw-50}$(2)$, we have 
$(x\ra y)\ra (x\Cap y)=(x\ra y)\ra (y\Cap x)=x$. 
Taking into consideration that $x\Cap y, y\Cap x\le_Q x\ra y$ (Proposition \ref{gqw-50}$(10)$), by  
Proposition \ref{qbe-30}$(1)$ (cancellation law), we get $x\Cap y=y\Cap x$. 
\end{proof}

\begin{lemma} \label{ciom-30} Let $x,y\in X$. The following hold: \\
$(1)$ if $x \mathcal{C} y$, then $x^* \mathcal{C} y^*;$ \\
$(2)$ if $x\le_Q y$ or $y\le_Q x$, then $x \mathcal{C} y;$ \\
$(3)$ $(x\Cap y)^* \mathcal{C} (x\ra y)^*;$ \\
$(4)$ $\mathcal{Z}(X)$ is closed under $^*$. 
\end{lemma}
\begin{proof}
$(1)$ It follows from the proof of Proposition \ref{ciom-20}$((a)\Rightarrow (b))$. \\
$(2)$ If $x\le_Q y$, then $x=x\Cap y$ and, by Proposition \ref{qbe-20}$(1)$ we have $x=y\Cap x$. 
It follows that $x\Cap y=y\Cap x$, hence $x \mathcal{C} y$. 
Similarly $y\le_Q x$ implies $y \mathcal{C} x$, that is $x \mathcal{C} y$. \\
$(3)$ Since $x\Cap y\le_Q y\le_Q x\ra y$, we get $(x\ra y)^*\le_Q (x\Cap y)^*$. 
Applying $(2)$, it follows that $(x\Cap y)^* \mathcal{C} (x\ra y)^*$. \\
$(4)$ Let $x\in \mathcal{Z}(X)$, and let $z\in X$, so that $x \mathcal{C} z^*$. 
Then we have $x^*\Cap z=(x\Cup z^*)^*=(z^*\Cup x)^*=z\Cap x^*$, and applying Proposition \ref{ciom-20} we get 
$x^* \mathcal{C} z$, that is $x^*\in \mathcal{Z}(X)$. 
\end{proof}

\begin{proposition} \label{ciom-50} If $x,y,z\in X$ such that $x \mathcal{C} y$ and $x \mathcal{C} z$, then 
$(x\Cup y)\Cup z=y\Cup (x\Cup z)$. 
\end{proposition}
\begin{proof} 
Using the commutativity and Proposition \ref{gqw-50-10}$(6)$,$(13)$, we get: \\
$\hspace*{2.10cm}$ $(x\Cup y)\Cup z=(y\Cup x)\Cup z=(y\Cup x)\Cup (x\Cup z)$ \\
$\hspace*{4.00cm}$ $=((y\Cup x)\ra (x\Cup z))\ra (x\Cup z)=((y\Cup x)\ra (z\Cup x))\ra (z\Cup x)$ \\
$\hspace*{4.00cm}$ $=(y\ra (z\Cup x))\ra (z\Cup x)=y\Cup (z\Cup x)=y\Cup (x\Cup z)$. 
\end{proof}

\begin{corollary} \label{ciom-60} If $x,y,z\in X$ such that $x \mathcal{C} y$, $y \mathcal{C} z$, $x \mathcal{C} z$, then: \\
$(1)$ $(x\Cup y)\Cup z=z\Cup (x\Cup y);$ \\
$(2)$ $y\Cup z\le_Q y\Cup (z\Cup x)$.  
\end{corollary}
\begin{proof}
$(1)$ By hypothesis and applying Proposition \ref{ciom-50}, we get: \\
$\hspace*{2.10cm}$ $(x\Cup y)\Cup z=y\Cup (x\Cup z)=y\Cup (z\Cup x)$ \\
$\hspace*{4.00cm}$ $=(z\Cup y)\Cup x=(y\Cup z)\Cup x=z\Cup (y\Cup x)=z\Cup (x\Cup y)$. \\
$(2)$ From $z\le_Q x\Cup z=z\Cup x$, we get $z\Cup y\le_Q (z\Cup x)\Cup y$, hence by $(1)$, 
$y\Cup z\le_Q y\Cup (z\Cup x)$. 
\end{proof}

\begin{corollary} \label{ciom-70} $\mathcal{Z}(X)$ is closed under $\Cup$ and $\Cap$. 
\end{corollary}
\begin{proof} 
Let $x,y \in \mathcal{Z}(X)$ and let $z \in X$, so that $x \mathcal{C} y$, $y \mathcal{C} z$, 
$x \mathcal{C} z$, and by Corollary \ref{ciom-60} we get $(x\Cup y)\Cup z=z\Cup (x\Cup y)$. 
Thus  $x\Cup y \in \mathcal{Z}(X)$, that is $\mathcal{Z}(X)$ is closed under $\Cup$. 
Since by Lemma \ref{ciom-30}, $\mathcal{Z}(X)$ is closed under $^*$ and $x\Cap y=(x^*\Cup y^*)^*$, 
it follows that \\
$\hspace*{2.10cm}$ $(x\Cap y)\Cap z=((x^*\Cup y^*)\Cup z^*)^*=(z^*\Cup (x^*\Cup z^*))^*=z\Cap (x\Cap y)$, \\ 
that is $\mathcal{Z}(X)$ is also closed under $\Cap$.
\end{proof}

\begin{proposition} \label{ciom-80} The following holds for all $x,y,z\in X$ such that $y \mathcal{C} z$: \\
$\hspace*{4cm}$  $(y\Cup z)\ra x\le_Q (y\ra x)\Cap (z\ra x)$. 
\end{proposition}
\begin{proof} 
Since $y\le_Q z\Cup y$, it follows that $(z\Cup y)\ra x\le_Q y\ra x$, hence $(y\ra x)^*\le_Q ((z\Cup y)\ra x)^*$. 
It follows that $((z\Cup y)\ra x)^*\ra (z\ra x)^*\le_Q (y\ra x)^*\ra (z\ra x)^*$, that is \\
$(\alpha)$ $(((z\Cup y)\ra x)^*\ra (z\ra x)^*)\Cap ((y\ra x)^*\ra (z\ra x)^*)=((z\Cup y)\ra x)^*\ra (z\ra x)^*$. \\
On the other hand, from $z\le_Q y\Cup z$, we get $(y\Cup z)\ra x\le_Q z\ra x$, so that \\ 
$(\beta)$ $((y\Cup z)\ra x)\Cap (z\ra x)=(y\Cup z)\ra x$. \\
Then we have: \\
$\hspace*{0.50cm}$ $(y\Cup z)\ra x=((y\Cup z)\ra x)\Cap (z\ra x)$ (by $\beta$) \\
$\hspace*{1.00cm}$ $=((((y\Cup z)\ra x)^*\ra (z\ra x)^*)\ra (z\ra x)^*)^*$ \\
$\hspace*{1.00cm}$ $=((((z\Cup y)\ra x)^*\ra (z\ra x)^*)\ra (z\ra x)^*)^*$ ($y \mathcal{C} z$) \\
$\hspace*{1.00cm}$ $=((((z\Cup y)\ra x)^*\ra (z\ra x)^*)\Cap ((y\ra x)^*\ra (z\ra x)^*))\ra (z\ra x)^*)^*$ 
                                                                                            (by $\alpha$) \\
$\hspace*{1.00cm}$ $=((((y\Cup z)\ra x)^*\ra (z\ra x)^*)\Cap ((y\ra x)^*\ra (z\ra x)^*))\ra (z\ra x)^*)^*$ 
                                                                                            ($y \mathcal{C} z$) \\
$\hspace*{1.00cm}$ $=(((z\ra x)\ra ((y\Cup z)\ra x))\Cap ((z\ra x)\ra (y\ra x))\ra (z\ra x)^*)^*$ \\
$\hspace*{1.00cm}$ $=((y\Cup z)\ra x)\Cap ((y\ra x)\Cap (z\ra x))$ (Prop. \ref{qbe-30}$(9)$). \\
Hence $(y\Cup z)\ra x\le_Q (y\ra x)\Cap (z\ra x)$. 
\end{proof}

\begin{proposition} \label{ciom-90} The following holds for all $x,y,z\in X$ such that $x \mathcal{C} y$, 
$y \mathcal{C} z$, $x \mathcal{C} z$: \\
$\hspace*{4cm}$  $(y\ra x)\Cap (z\ra x)\le (y\Cup z)\ra x$. 
\end{proposition}
\begin{proof} 
Since by Corollary \ref{ciom-60}, $y\Cup z\le_Q y\Cup (z\Cup x)$, using Proposition \ref{qbe-20}$(4)$ we get 
$y\Cup z\le y\Cup (z\Cup x)$. Then we have: \\
$\hspace*{2.20cm}$ $1=(y\Cup z)\ra (y\Cup (z\Cup x))$ \\
$\hspace*{2.50cm}$ $=(y\ra (z\Cup x))\ra ((y\Cup z)\ra (z\Cup x))$ (Prop. \ref{qbe-20}$(8)$) \\
$\hspace*{2.50cm}$ $=((z\ra x)\ra (y\ra x))\ra ((z\ra x)\ra ((y\Cup z)\ra x))$ (Prop. \ref{qbe-20}$(8)$) \\
$\hspace*{2.50cm}$ $=((y\ra x)^*\ra (z\ra x)^*)\ra (((y\Cup z)\ra x)^*\ra (z\ra x)^*)$ \\
$\hspace*{2.50cm}$ $=((y\Cup z)\ra x)^*\ra (((y\ra x)^*\ra (z\ra x)^*)\ra (z\ra x)^*)$ \\
$\hspace*{2.50cm}$ $=((y\Cup z)\ra x)^*\ra ((y\ra x)^*\Cup (z\ra x)^*)$ \\
$\hspace*{2.50cm}$ $=((y\Cup z)\ra x)^*\ra ((y\ra x)\Cap (z\ra x))^*$ \\
$\hspace*{2.50cm}$ $=((y\ra x)\Cap (z\ra x))\ra ((y\Cup z)\ra x)$. \\
Hence $(y\ra x)\Cap (z\ra x)\le (y\Cup z)\ra x$. 
\end{proof}

\begin{theorem} \label{ciom-100} The following hold for all $x,y,z\in X$ such that $x \mathcal{C} y$, 
$y \mathcal{C} z$, $x \mathcal{C} z$: \\
$(1)$ $(y\Cup z)\ra x=(y\ra x)\Cap (z\ra x);$ \\
$(2)$ $x\ra (y\Cap z)=(x\ra y)\Cap (x\ra z)$.  
\end{theorem}
\begin{proof} 
$(1)$ It follows by Propositions \ref{ciom-80}, \ref{ciom-90}, \ref{gqw-50}$(7)$. \\
$(2)$ Using $(1)$, we get: \\
$\hspace*{2.00cm}$ $x\ra (y\Cap z)=x\ra (y^*\Cup z^*)^*=(y^*\Cup z^*)\ra x^*$ \\
$\hspace*{4.00cm}$ $=(y^*\ra x^*)\Cap (z^*\ra x^*)=(x\ra y)\Cap (x\ra z)$. 
\end{proof}

\begin{lemma} \label{ciom-110} The following hold: \\
$(1)$ if $y\in \mathcal{Z}(X)$ and $x,z\in X$, then \\  
$\hspace*{2cm}$ $(z\Cap x)^*\ra ((z\ra x)^*\Cap y)=((z\Cap x)^*\ra (z\ra x)^*)\Cap ((z\Cap x)^*\ra y);$ \\
$(2)$ if $x,y,z\in \mathcal{Z}(X)$, then \\ 
$\hspace*{2cm}$ $(z\Cap x)^*\ra y=(y^*\ra z)\Cap (x^*\ra y);$ \\
$(3)$ if $x,y\in \mathcal{Z}(X)$ and $z\in X$, then \\
$\hspace*{2.00cm}$ $z=(z\ra x)\ra (z\Cap x)$ and \\
$\hspace*{2.00cm}$ $(z\ra x)^*=((z\ra x)^*\ra y)\ra ((z\ra x)^*\Cap y)$.  
\end{lemma}
\begin{proof} 
$(1)$ Since $y\in \mathcal{Z}(X)$, we have $y\mathcal{C} (z\Cap x)^*$ and $y\mathcal{C} (z\ra x)^*$. 
By Lemma \ref{ciom-30}$(3)$, $(z\Cap x)^* \mathcal{C} (z\ra x)$, and we apply Theorem \ref{ciom-100}$(2)$ 
for $(z\Cap x)^*$, $(z\ra x)^*$ and $y$. \\
Theorem \ref{ciom-100}$(2)$, since 
$y\mathcal{C} (z\Cap x)^*$ and $y\mathcal{C} (z\ra x)^*$. \\
$(2)$ By Lemma \ref{ciom-30}$(4)$, $y\in \mathcal{Z}(X)$ implies $y^*\in \mathcal{Z}(X)$, and applying 
Theorem \ref{ciom-100}$(2)$, we get: 
$(z\Cap x)^*\ra y=y^*\ra (z\Cap x)=(y^*\ra z)\Cap (y^*\ra x)=(y^*\ra z)\Cap (x^*\ra y)$. \\
$(3)$ Since $x\mathcal{C} z$ and $y\mathcal{C} (z\ra x)^*$, appying Proposition \ref{ciom-20}$(3)$, we get: 
$z=(z\ra x)\ra (z\Cap x)$ and $(z\ra x)^*=((z\ra x)^*\ra y)\ra ((z\ra x)^*\Cap y)$.  
\end{proof}

\begin{proposition} \label{ciom-120} If $x,y\in \mathcal{Z}(X)$, then $x\ra y\in \mathcal{Z}(X)$. 
\end{proposition}
\begin{proof} 
Based on the previous results, we have: \\
$\hspace*{0.50cm}$ $(z\ra (x^*\ra y))\ra (z\Cap (x^*\ra y))$ \\
$\hspace*{2.00cm}$ $=((z\ra x)^*\ra y)\ra (z\Cap (x^*\ra y))$ (Lemma \ref{qbe-10}$(7)$) \\
$\hspace*{2.00cm}$ $=((z\ra x)^*\ra y)\ra (z\Cap ((y^*\ra z)\Cap (x^*\ra y))$ (Prop. \ref{gqw-50-10}$(11)$) \\
$\hspace*{2.00cm}$ $=((z\ra x)^*\ra y)\ra (z\Cap ((z\Cap x)^*\ra y))$ (Lemma \ref{ciom-110}$(2)$) \\
$\hspace*{2.00cm}$ $=((z\ra x)^*\ra y)\ra ((z\Cup (z\ra x)^*)\Cap ((z\Cap x)^*\ra y))$ 
                                                                           (Prop. \ref{gqw-50-10}$(12)$) \\
$\hspace*{2.00cm}$ $=((z\ra x)^*\ra y)\ra (((x\Cap z)^*\ra (z\ra x)^*)\Cap ((z\Cap x)^*\ra y))$ 
                                                                          (Prop. \ref{qbe-30}$(10)$) \\
$\hspace*{2.00cm}$ $=((z\ra x)^*\ra y)\ra (((z\Cap x)^*\ra (z\ra x)^*)\Cap ((z\Cap x)^*\ra y))$ 
                                                                          (since $x\mathcal{C} z$) \\
$\hspace*{2.00cm}$ $=((z\ra x)^*\ra y)\ra ((z\Cap x)^*\ra ((z\ra x)^*\Cap y))$ (Lemma \ref{ciom-110}$(1)$) \\
$\hspace*{2.00cm}$ $=((z\ra x)^*\ra y)\ra (((z\ra x)^*\Cap y)^*\ra (z\Cap x))$ \\
$\hspace*{2.00cm}$ $=(((z\ra x)^*\ra y)\ra ((z\ra x)^*\Cap y))^*\ra (z\Cap x)$ (Lemma \ref{qbe-10}$(7)$) \\
$\hspace*{2.00cm}$ $=((z\ra x)^*)^*\ra (z\Cap x)$ (Lemma \ref{ciom-110}$(3)$) \\
$\hspace*{2.00cm}$ $=(z\ra x)\ra (z\Cap x)$ \\ 
$\hspace*{2.00cm}$ $=z$ (Lemma \ref{ciom-110}$(3)$). \\
From $(z\ra (x^*\ra y))\ra (z\Cap (x^*\ra y))=z$, applying Proposition \ref{ciom-20}$(3)$, it follows that 
$(x^*\ra y))\mathcal{C} z$, for any $z\in X$, hence $x^*\ra y\in \mathcal{Z}(X)$. \\
Since $x\in \mathcal{Z}(X)$ implies $x^*\in \mathcal{Z}(X)$, we have $x\ra y=(x^*)^*\ra y\in \mathcal{Z}(X)$. 
\end{proof}

\begin{theorem} \label{ciom-130} $(\mathcal{Z}(X),\ra,^*,1)$ is a Wajsberg subalgebra of $X$.    
\end{theorem}
\begin{proof} 
By Proposition \ref{ciom-120}, $\mathcal{Z}(X)$ is closed under $\ra$. 
Moreover $0,1\in \mathcal{Z}(X)$, hence it is a quantum-Wajsberg subalgebra of $X$. 
Since by Corollary \ref{ciom-70}, $x,y\in \mathcal{Z}(X)$ implies $x\Cup y=y\Cup x$, 
$(\mathcal{Z}(X),\ra,^*,1)$ is a commutative quantum-Wajsberg algebra, that is a Wajsberg subalgebra of $X$.  
\end{proof}

Taking into consideration the above results, the commutative center $\mathcal{Z}(X)$ will be also called the 
\emph{Wajsberg-center} of $X$. 

\begin{corollary} \label{ciom-140} A quantum-Wajsberg algebra $X$ is a Wajsberg algebra if and only if  $\mathcal{Z}(X)=X$.    
\end{corollary}

$\vspace*{5mm}$

\section{Logical aspects}
In this section we investigate certain properties of logics based on the new algebras introduced in the paper. 

\S{\bf 1.} The behavior of physical systems which describe the events concerning elementary particles is different 
from that of physical systems observed in classical physics.
The quantum effects for a physical system are usually described by the set $\mathcal{E}(H)$ of all self-adjoint 
operators between null operator $\mathbb{O}$ and identity operator $\mathbb{I}$ on $H$. 
The set $\mathcal{E}(H)$ of Hilbert space effects is the prototype of the abstract definition of effect algebras 
which were introduced by R. Giuntini and H. Greuling under the name of weak orthoalgebras (\cite{Giunt8}), as an algebraic axiomatisation of the logic of quantum mechanics (see also \cite{Foulis}). 

\begin{definition} \label{liom-10} $\rm($\cite[Def. 1.2.1]{DvPu}$\rm)$ \emph{
An \emph{effect algebra} is a system $(E,\oplus,0,1)$ consisting of a set $E$ with two special elements 
$0,1\in E$, called the \emph{zero} and \emph{unit}, and with a partially defined operation $\oplus$ satisfying the following conditions for all $x,y,z\in E$: \\
$(E_1)$ (Commutative law) If $x\oplus y$ is defined, then $y\oplus x$ is defined and $x\oplus y=y\oplus x;$ \\
$(E_2)$ (Associative law) If $y\oplus z$ and $x\oplus (y\oplus z)$ are defined, then $x\oplus y$ and 
$(x\oplus y)\oplus z$ are defined, and $(x\oplus y)\oplus z=x\oplus (y\oplus z);$ \\
$(E_3)$ (Orthosupplement law) For each $x\in E$ there exists a unique $y\in E$ such that $x\oplus y$ is defined 
and $x\oplus y=1;$ \\
$(E_4)$ (Zero-one law) If $x\oplus 1$ is defined, then $x=0$. 
}
\end{definition}

\begin{proposition} \label{liom-20} Let $(X,\ra,^*,0,1)$ be an implicative-orthomodular algebra. 
Define the operation $\oplus$ on $X$ by $x\oplus y=x^*\ra y$, whenever $x\le_Q y^*$. 
Then the structure $(X,\oplus,0,1)$ is an effect algebra. 
\end{proposition}
\begin{proof} We check axioms $(E_1)$-$(E_4)$ from the definition of an effect algebra. \\
$(E_1)$ Since $x\oplus y$ is defined, then $x\le_Q y^*$ and $y\le_Q x^*$, that is $y\oplus y$ is also defined. 
Obviously, $x\oplus y=x^*\ra y=y^*\ra x=y\oplus x$. \\
$(E_2)$ Assume $y\oplus z$ and $x\oplus (y\oplus z)$ are defined, so that $y\le_Q z^*$ and 
$x\le (y\oplus z)^*=(y^*\ra z)^*$. 
From $y\le_Q y^*\ra z$, we get $(y^*\ra z)^*\le_Q y^*$, hence $x\le_Q y^*$, that is $x\oplus y$ is defined. 
Moreover, $x\le_Q (y^*\ra z)^*$ implies $y^*\ra z\le_Q x^*$, so that $x^*\ra y\le_Q(y^*\ra z)\ra y=z^*\Cup y$. 
Since $y\le z^*$, we have $y=y\Cap z^*$, and applying Proposition \ref{gqw-50-10}$(2)$, we get  
$x^*\ra y\le_Q z^*\Cup y=z^*\Cup (y\Cap z^*)=z^*$. Hence $x^*\ra y\le z^*$, that is $(x\oplus y)\oplus z$ is 
defined. 
Finally, using Lemma\ref{qbe-10}$(7)$, we have 
$(x\oplus y)\oplus z=(x^*\ra y)^*\ra z=x^*\ra (y^*\ra z)=x\oplus (y\oplus z)$. \\
$(E_3)$ Taking $y:=x^*$, we have $x\le_Q x=(x^*)^*=y^*$, that is $x\oplus y$ is defined, and 
$x\oplus y=x^*\ra x^*=1$. 
Assume that there exists another $y_1\in E$ such that $x\oplus y_1$ is defined and $x\oplus y=x\oplus y_1$. 
From $x\le_Q y^*$ and $x\le_Q y_1^*$, we have $y,y_1\le_Q x^*$, while $x\oplus y=x\oplus y_1$ implies 
$x^*\ra y=x^*\ra y_1$. Applying the cancellation law, it follows that $y=y_1$. \\
$(E_4)$ If $x\oplus 1$ is defined, then $x\le_Q 1^*=0$, that is $x=0$. 
\end{proof}

\begin{remark} \label{liom-30} The meta-Wajsberg algebras and the pre-Wajsberg algebras do not have corresponding  effect algebras. \\
Indeed, let $X$ be the meta-Wajsberg algebra from Example \ref{gqw-180}. 
We have $b\le_Q d^*$, so that $b\oplus d$ exists, while $d\nleq_Q b^*$, that is $d\oplus b$ does not exist. 
It follows that $(X,\oplus,0,1)$ is not an effect algebra. 
Moreover, by Proposition \ref{gqw-20} any pre-Wajsberg algebra is a meta-Wajsberg algebra, so that, in general, 
the meta-Wajsberg algebras do not have corresponding  effect algebras. 
\end{remark}

\begin{corollary} \label{liom-40} 
$(1)$ An implicative-orthomodular algebra can be an algebraic model for the logic of quantum mechanics  
(within the framework of effects in a Hilbert space); \\
$(2)$ The pre-Wajsberg and meta-Wajsberg algebras are not algebraic models for the logic of quantum mechanics. 
\end{corollary}

\S{\bf 2.} According to Birkhoff's HSP theorem (\cite{Birk}), a class of algebraic structures wih the same signature 
is a variety if and only if it is closed under \emph{HSP operations} ({\bf H}omomorphism, {\bf S}ubalgebra, 
{\bf P}roduct). 
If $\mathbb{K}$ is a class of algebraic structures of the same type, the variety generated by $\mathbb{K}$ will be denoted by $HSP(\mathbb{K})$ (\cite{Giunt3}).  
Denote by $\mathbb{QW}$, $\mathbb{PW}$, $\mathbb{MW}$ and $\mathbb{IOM}$ the classes of all quantum-Wajsberg, pre-Wajsberg, meta-Wajsberg and implicative-orthomodular algebras, respectively. 
The logics based on $\mathbb{QW}$, $\mathbb{PW}$, $\mathbb{MW}$, $\mathbb{IOM}$ are clearly axiomatizable, since 
$HSP(\mathbb{QW})=\mathbb{QW}$, $HSP(\mathbb{PW})=\mathbb{PW}$, $HSP(\mathbb{MW})=\mathbb{MW}$, $HSP(\mathbb{IOM})=\mathbb{IOM}$. \\

\S{\bf 3.} According to \cite[Th. 4.7]{Blok}, a logic  $\mathcal{S}$ over the language $\mathcal{L}$ is \emph{algebraizable} if it verifies the following properties, for all formulas 
$\varphi, \phi, \vartheta, \varphi_1, \phi_1$ in $\mathcal{L}$: \\
$(i)$ $\vdash\varphi\Delta\varphi$, \\
$(ii)$ $\varphi\Delta\phi\vdash\phi\Delta\varphi$, \\
$(iii)$ $\varphi\Delta\phi$,$\phi\Delta\vartheta\vdash\varphi\Delta\vartheta$, \\
$(iv)$ $\varphi\Delta\phi,\varphi_1\Delta\phi_1\vdash(\varphi\ra \varphi_1)\Delta(\phi\ra \phi_1)$, \\
$(v)$ $\varphi\dashv\vdash\varphi\Delta(\varphi\ra \varphi)$, \\
where $\varphi\Delta\phi$ is an abbreviation of $\{\varphi\ra \phi, \phi\ra \varphi\}$ (see also \cite{Kuhr5}). \\
It was proved in \cite{Ciu78} that the class of bounded commutative BCK algebras is a subvariety of the variety of 
quantum-Wajsberg algebras, that is a subvariety of the varieties of pre-Wajsberg, meta-Wajsberg and 
implicative-orthomodular algebras. 
J. $\rm K \ddot{u}hr$ proved in \cite[Prop. 1.3.2]{Kuhr5} that the BCK logic is algebraizable in the sense of \cite{Blok}. 
As a consequence, the logic based on bounded commutative BCK algebras (i.e., Wajsberg algebras) is algebraizable. 
For the logics based on quantum-Wajsberg, pre-Wajsberg, meta-Wajsberg and implicative-orthomodular algebras, we leave as an open problem whether they are algebraizable or not.

$\vspace*{5mm}$

\section{Concluding remarks}

In this paper, we redefined the notions of pre-MV and meta-MV algebras, and we introduced and studied the pre-Wajsberg and meta-Wajsberg algebras. 
We established certain connections between these structures: every pre-Wajsberg algebra is a meta-Wajsberg algebra, 
and an implicative-orthomodular algebra is a quantum-Wajsberg algebra if and only if it is a meta-Wajsberg algebra. 
Other connections were proved based on the implicative condition of BE algebras. \\ 
A. Iorgulescu clarified in \cite{Ior32} some aspects concerning the quantum-MV algebras as non-idempotent generalizations of orthomodular lattices: she studied the orthomodular lattices and she introduced and studied two generalizations of them, the orthomodular softlattices and the orthomodular widelattices. 
More precisely, an involutive m-BE algebra $(X,\odot,^*,1)$ is called (\cite[Def. 3.2, 4.2, 4.19]{Ior32}): \\
$\hspace*{0.50cm}$ $-$ an \emph{orthomodular lattice} if it satisfies: \\
$\hspace*{0.85cm}$ $(Pom)$ $(x\odot y)\oplus ((x\odot y)^*\odot x)=x$, or, equivalently, $x\Cup (x\odot y)=x$, \\
$\hspace*{0.85cm}$ $(m$-$Pimpl)$ $((x\odot y^*)^*\odot x^*)^*=x;$ \\
$\hspace*{0.50cm}$ $-$ an \emph{orthomodular softlattice} if it satisfies $(Pom)$ and: \\
$\hspace*{0.85cm}$ $(G)$ $x\odot x=x;$ \\
$\hspace*{0.50cm}$ $-$ an \emph{orthomodular widelattice} if it satisfies $(Pom)$ and: \\
$\hspace*{0.85cm}$ $(m$-$Pabs$-$i)$ $x\odot (x\oplus x\oplus x\oplus y)=x$. \\ 
We redefined the orthomodular lattices, by introducing the notion of \emph{implicative-orthomodular lattices} as involutive BE algebras satisfying conditions $(QW_2)$ and $(Pimpl)$, or, equivalently, implicative-orthomodular algebras satisfying condition $(Pimpl)$. \\
In a future work, we will define the notions of \emph{implicative-orthomodular softlattices} and 
\emph{implicative-orthomodular widelattices}, and we will investigate their connections to quantum-Wajsberg 
algebras. \\ 
{\bf Open problem.} Investigate the algebraizability and protoalgebraizability (in the sense of \cite{Blok, Font2}) of logics based on quantum-Wajsberg, pre-Wajsberg, meta-Wajsberg and implicative-orthomodular algebras. \\

$\vspace*{1mm}$
          
\begin{center}
\sc Acknowledgement 
\end{center}
The author is very grateful to the anonymous referees for their useful remarks and suggestions on the subject that helped improving the presentation.

$\vspace*{1mm}$

\end{document}